\documentclass[12pt]{amsart}
\usepackage{amsmath,amssymb,amsthm,mathtools}
\usepackage{amsfonts}
\usepackage{esint} 

\usepackage{mathrsfs}
\usepackage{enumerate}
\usepackage[top=1.3in, bottom=1.2in, left=1in, right=1in]{geometry}
\usepackage[english]{babel}
\usepackage[
    colorlinks=true,
    linkcolor=blue,
    citecolor=blue,
    urlcolor=blue,
    backref=page
]{hyperref}
\renewcommand*{\backref}[1]{}
\renewcommand*{\backrefalt}[4]{%
  \ifcase #1
  \or
    $\uparrow$ #2%
  \else
    $\uparrow$ #2%
  \fi
}

\usepackage{enumitem}
\usepackage[nameinlink,capitalize]{cleveref}

\numberwithin{equation}{section}

\newtheorem{theorem}{Theorem}[section]

\newtheorem{example}{Example}[section]
\newtheorem{proposition}[theorem]{Proposition}
\newtheorem{corollary}[theorem]{Corollary}
\newtheorem{lemma}[theorem]{Lemma}
\newtheorem{remark}{Remark}[section]

\newtheorem{assumption}{Assumption}[section]

\newcommand{\Ric}{\textup{Ric}}

\newcommand{\CP}{\mathbb{CP}}

\newcommand{\Id}{\operatorname{Id}}

\newcommand{\tr}{\textup{tr}}

\def\RR{\mathbb{R}}
\def\CC{\mathbb{C}}
\def\HH{\mathbb{H}}

\def\HP{\mathbb{HP}}
\def\CP{\mathbb{CP}}
\def\K{\mathbb{K}}
\def\KP{\mathbb{KP}}

\makeatletter
\newcommand*\owedge{\mathpalette\@owedge\relax}
\newcommand*\@owedge[1]{%
  \mathbin{%
    \ooalign{%
      $#1\m@th\bigcirc$\cr
      \hidewidth$#1\m@th\wedge$\hidewidth\cr
    }%
  }%
}
\makeatother

\setlist[enumerate]{
itemsep = 0.3em
}

\begin{document}

\title[Sharp focal radius estimate and rigidity of hypersurfaces]{Sharp focal radius estimate and rigidity of hypersurfaces in manifolds with positive curvature}
\author{Tsz-Kiu Aaron Chow}
\address{Department of Mathematics, Hong Kong University of Science and Technology, Hong Kong S.A.R., China}
\email{\href{chowtka@ust.hk}{chowtka@ust.hk}}
\thanks{T.-K. A. C. is supported  by the Croucher Foundation Start-up Grant and the HKUST New Faculty Start-up Grant}

\author{Jingbo Wan}
\address{Laboratoire Jacques-Louis Lions de Sorbonne Universit\'e, 4 place Jussieu, Paris 75005, France}
\email{\href{jingbo.wan@sorbonne-universite.fr}{jingbo.wan@sorbonne-universite.fr}}
\thanks{J. W. is supported by ERC-2023 AdG 101141855 BLaHST}

\begin{abstract}
We prove a sharp Clifford-threshold focal-radius estimate and rigidity for immersed
hypersurfaces.  Under a $p$-form curvature condition, formulated by the
Weitzenb\"ock curvature term together with
$\Ric_p\ge p$, any closed two-sided immersion
$F:\Sigma^m\to M^{m+1}$ with $b_p(\Sigma;\mathbb R)\neq0$ and
$1\le p\le m/2$ satisfies
\[
    r_f(F,M)\le\frac{\pi}{4}.
\]
The equality case is rigid: if the ambient manifold is complete, equality
forces the hypersurface to be locally the Clifford hypersurface
$S^p(1/\sqrt2)\times S^{m-p}(1/\sqrt2)\subset S^{m+1}(1)$; if the ambient
manifold is compact and connected, it is a spherical space form.  The
curvature condition follows from $\sec\ge1$ for $p=1$, from normalized
$\mathrm{PIC1}\ge1$ for $p=2$, and from curvature operator bounded below by
one in all degrees.  By quotient lifting and the Hopf fibrations, we also
obtain focal-radius estimates in $\CP^n$ and $\HP^n$, with
projective Clifford rigidity, without any Betti-number assumption.
\end{abstract}

\maketitle


\section{Introduction}
\label{sec:introduction}

Focal radius is an extrinsic counterpart of diameter: it measures the
largest interval on which the normal exponential map of a submanifold remains
nonsingular.  Under positive curvature, focal-radius bounds often reflect the
interaction between ambient curvature and the topology of the submanifold.
The basic sharp threshold is $\pi/2$.  A theorem of Guijarro--Wilhelm
\cite{GuijarroWilhelm2018} shows that, in sectional curvature at least one,
the focal radius is at most $\pi/2$ under natural closedness assumptions, and
equality corresponds to the totally geodesic model.

The next natural model is the Clifford hypersurface
\[
    S^p\left(\frac1{\sqrt2}\right)
    \times
    S^{m-p}\left(\frac1{\sqrt2}\right)
    \subset S^{m+1}(1).
\]
Its two focal submanifolds lie at distance $\pi/4$ on the two sides.  Unlike
the totally geodesic sphere, this hypersurface carries nontrivial intermediate
cohomology.  Thus the Clifford example suggests a sharper question: can
nontrivial $p$-th cohomology force the focal radius down from the totally
geodesic threshold $\pi/2$ to the Clifford threshold $\pi/4$? We prove that the answer is yes under curvature condition tailored to the
Bochner formula on $p$-forms.

The curvature condition has two parts.  First, the ambient curvature restricted
to each tangent hyperplane has sufficiently positive Weitzenb\"ock curvature
on $p$-forms.  Second, the intermediate Ricci curvature satisfies
$\Ric_p\ge p$.  These assumptions have the same sharp normalization and combine
with the normal index form to give a Bochner obstruction to the existence of two focal-free normal directions of length greater than $\pi/4$.

The equality case is rigid.  The equality analysis first gives a
Clifford-type splitting of the pullback normal cylinder.  Since the ambient
normal exponential map extends smoothly to the focal points when $M$ is
complete, the collapsed factors are automatically round.  Hence the
hypersurface is locally Clifford.  If $M$ is compact and connected, the open
Clifford collar is dense, so the ambient metric has constant sectional
curvature one.

The assumptions include familiar positivity conditions.  For
$p=1$, it follows from $\sec\ge1$.  For $p=2$, it follows from normalized
$\mathrm{PIC1}\ge1$.  If the curvature operator is bounded below by the
identity, it holds in every degree.  Thus the result is a
Clifford-threshold analogue of the Bochner vanishing philosophy of Bochner,
Gallot--Meyer, and Petersen--Wink
\cite{Bochner1946,GallotMeyer1975,PetersenWink2020}, with an additional
extrinsic comparison input.

This $\pi/4$ phenomenon is closely related to recent results in scalar and
positive curvature geometry.  Gromov's metric inequalities with scalar
curvature \cite{Gromov2018} and the theory of torical and Clifford bands
highlight Clifford products as sharp extremal models.  In dimension three,
Zhu proved that if $(M^3,g)$ is a $3$-sphere with $\sec\ge1$ and
$\Sigma\to M$ is an immersed torus, then
$r_f(\Sigma,M)\le\pi/4$, with equality only for the round metric and the
Clifford torus \cite{Zhu2021}.  Ge proved a sharp normal-injectivity-radius
estimate for embedded orientable hypersurfaces in the round sphere which are
not topological spheres, again with Clifford hypersurfaces as equality models
\cite{Ge2024}.  Hirsch--Kazaras--Khuri--Zhang proved sharp band-width and
normal-injectivity-radius estimates under positive Ricci and positive
$2$-Ricci curvature assumptions using spacetime harmonic functions
\cite{HirschKazarasKhuriZhang2025}.  In the positive isotropic curvature
setting, Chow--Zhu obtained related focal-radius bounds under PIC and
Betti-number assumptions \cite{ChowZhu2024}.

The present theorem extends this \(\pi/4\) picture in a different direction.  It estimates focal radius rather
than normal injectivity radius, so it applies naturally to immersed
hypersurfaces.  It works in arbitrary degree $p$, with the curvature condition
matched to the Bochner formula on $p$-forms.  Finally, its equality case
identifies the local ambient geometry, not only the hypersurface model.

Another application of the same method comes from quotient geometry.  Although the Betti-number
hypothesis is not stable under quotients, the lift of a hypersurface to an
odd-dimensional total space gains topology from the vertical directions.
Horizontal lifting preserves normal focal radii.  Hence an odd-dimensional
total space satisfying the curvature assumptions in all degrees gives the
same $\pi/4$ bound for every hypersurface in a free positive-dimensional
isometric quotient.  Applying this to the Hopf fibrations gives focal-radius
estimates for immersed hypersurfaces in $\CP^n$, $\HP^n$  without
any Betti-number assumption, with projective Clifford rigidity in the complex
and quaternionic cases.

\subsection{Setup and main results}
\label{subsec:setup-main-results}

Let $(M^{m+1},g_M)$, $m\ge2$, be a Riemannian manifold.  We adopt the convention
\[
    R(X,Y)Z=\nabla_X\nabla_YZ-\nabla_Y\nabla_XZ-\nabla_{[X,Y]}Z,
    \qquad
    R_{ijkl}=-g(R(e_i,e_j)e_k,e_l),
\]
so that $K(e_i\wedge e_j)=R_{ijij}$.  For symmetric two-tensors $h,k$, recall the Kulkarni-Nomizu product
\[
\begin{aligned}
    (h\owedge k)(X,Y,Z,W)
    =\frac12\{&h(X,Z)k(Y,W)+k(X,Z)h(Y,W) \\
              &-h(X,W)k(Y,Z)-k(X,W)h(Y,Z)\}.
\end{aligned}
\]
Thus $g\owedge g$ has constant sectional curvature $1$.

Let $F:\Sigma^m\to M^{m+1}$ be a closed, connected, two-sided immersed hypersurface, and
choose a unit normal field $\nu$.  Define
\[
    \Phi(x,t)=\exp_{F(x)}(t\nu_x).
\]
The two-sided focal radius is
\[
    r_f(F,M)
    =
    \sup\left\{
    r>0:
    \Phi_t \text{ is defined and nonsingular for every } |t|<r
    \right\}.
\]
For $0<r<r_f(F,M)$, set $\mathcal C_r=\Sigma\times(-r,r)$ and
$g=\Phi^*g_M$.  We write $g_t=g|_{\Sigma\times\{t\}}$ and
$g_0=F^*g_M$.  The second fundamental form convention is
\[
    A(X,Y)=g(\nabla_X\nu,Y),
    \qquad
    Ae_i=\kappa_i e_i,
    \qquad
    H=\sum_{i=1}^m\kappa_i.
\]

Fix $1\le p\le m/2$.  Let $q^{(p)}$ be the Weitzenb\"ock curvature term on
$p$-forms:
\[
    \Delta_H\omega=(d+d^*)^2\omega=-\Delta\omega+q^{(p)}\omega,
    \qquad
    q^{(p)}(R)\omega
    =
    -\sum_{i,j}e^i\wedge i_{e_j}(R(e_i,e_j)\omega).
\]
By the Gauss equation,
\[
    R^\Sigma=R_T^M+A\owedge A,
\]
and hence, by linearity of $q^{(p)}$ in the curvature tensor,
\begin{equation}\label{eq:gp-Gauss-split}
    q_\Sigma^{(p)}
    =
    q_T^{M,p}+q_A^{(p)},
    \qquad
    q_T^{M,p}=q^{(p)}(R_T^M),
    \qquad
    q_A^{(p)}=q^{(p)}(A\owedge A).
\end{equation}
This decomposition also appears in \cite{Savo2014}.

\begin{assumption}[Ambient $p$-Bochner curvature]
\label{ass:gp-ambient-p-Bochner}
For every $x\in M$ and every $m$-plane $W\subset T_xM$,
\begin{equation}\label{eq:gp-ambient-p-Bochner}
    q^{(p)}(R^M|_W)\ge p(m-p)
    \qquad\text{on }\Lambda^pW^*.
\end{equation}
\end{assumption}

\begin{assumption}[Intermediate Ricci lower bound]
\label{ass:gp-Ric-p}
We assume $\Ric_p\ge p$, namely
\begin{equation}\label{eq:gp-Ric-p}
    \sum_{a=1}^p R(E_a,X,E_a,X)\ge p
\end{equation}
for every unit vector $X$ and every orthonormal $p$-frame
$E_1,\ldots,E_p\perp X$.
\end{assumption}

\begin{theorem}[Focal radius estimate]
\label{thm:gp-focal-estimate}
Let $1\le p\le m/2$.  Assume that $(M^{m+1},g_M)$ satisfies Assumptions
\ref{ass:gp-ambient-p-Bochner} and \ref{ass:gp-Ric-p} in degree $p$.
Let $F:\Sigma^m\to M^{m+1}$ be a closed two-sided immersed hypersurface.
If $b_p(\Sigma;\mathbb R)\neq0$, then
\[
    r_f(F,M)\le \frac{\pi}{4}.
\]
\end{theorem}

\begin{theorem}[Rigidity]
\label{thm:gp-rigidity}
Let $1\le p\le m/2$. Assume Assumptions
\ref{ass:gp-ambient-p-Bochner} and \ref{ass:gp-Ric-p}. Let
$F:\Sigma^m\to M^{m+1}$ be a closed connected two-sided immersed hypersurface
with $b_p(\Sigma;\mathbb R)\neq0$. Assume that $M$ is connected and complete.
If
\[
    r_f(F,M)=\frac{\pi}{4},
\]
then $(M,g_M)$ is a spherical space form of sectional curvature $1$. In
particular, its universal cover is isometric to $S^{m+1}(1)$, and every local
lift of $F$ to the universal cover is locally congruent to the standard
Clifford hypersurface
\[
    S^p\left(\frac1{\sqrt2}\right)
    \times
    S^{m-p}\left(\frac1{\sqrt2}\right)
    \subset S^{m+1}(1).
\]\end{theorem}

\subsection{Concrete curvature hypotheses}
\label{subsec:concrete-curvature-hypotheses}

We record several concrete curvature hypotheses under which Assumptions
\ref{ass:gp-ambient-p-Bochner} and \ref{ass:gp-Ric-p} hold.  Let $\mathcal C_{\mathrm{PIC1}}(V)$ be the cone
of algebraic curvature tensors $S$ such that
\begin{equation}\label{eq:PIC1-cone}
    S_{1313}
    +\lambda^2S_{1414}
    +S_{2323}
    +\lambda^2S_{2424}
    -2\lambda S_{1234}\ge0
\end{equation}
for every orthonormal four-frame and every $\lambda\in[0,1]$.  With our
normalization, $\mathrm{PIC1}\ge1$ means
\begin{equation}\label{eq:normalized-PIC1}
    R^M-g_M\owedge g_M\in\mathcal C_{\mathrm{PIC1}}(T_xM)
    \qquad\text{for every }x\in M.
\end{equation}

\begin{corollary}[The $b_1$-case]
\label{cor:p-one-sec}
Let $F:\Sigma^m\to M^{m+1}$ be a closed two-sided immersed hypersurface,
$m\ge2$.  If $\sec_M\ge1$ and $b_1(\Sigma;\mathbb R)\neq0$, then
\[
    r_f(F,M)\le\frac{\pi}{4}.
\]
If equality holds and $M$ is complete, then the rigidity conclusions of
Theorem \ref{thm:gp-rigidity} hold with $p=1$.
\end{corollary}

\begin{corollary}[The $b_2$-case]
\label{cor:p-two-PIC1}
Let $F:\Sigma^m\to M^{m+1}$ be a closed two-sided immersed hypersurface,
$m\ge4$.  Assume $\mathrm{PIC1}\ge1$ in the normalized sense
\eqref{eq:normalized-PIC1}.  If $b_2(\Sigma;\mathbb R)\neq0$, then
\[
    r_f(F,M)\le\frac{\pi}{4}.
\]
If equality holds and $M$ is complete, then the rigidity conclusions of
Theorem \ref{thm:gp-rigidity} hold with $p=2$.
\end{corollary}

\begin{remark}
For $p=1$, Corollary \ref{cor:p-one-sec} recovers the
$S^1\times S^{m-1}$ case of Zhu's Clifford focal-radius phenomenon
\cite{Zhu2021}, with the topological assumption weakened to
$b_1(\Sigma;\mathbb R)\neq0$.  Corollary \ref{cor:p-two-PIC1} gives the
corresponding $p=2$ analogue under normalized $\mathrm{PIC1}\ge1$.
\end{remark}

The quotient application is stated in Section \ref{sec:projective-spaces}.  If
an odd-dimensional total space $E^{2N+1}$ satisfies Assumptions
\ref{ass:gp-ambient-p-Bochner} and \ref{ass:gp-Ric-p} for every
$1\le p\le N$, for instance if $\mathcal R_E\ge\Id_{\Lambda^2TE}$, then every
closed connected immersed hypersurface in a free positive-dimensional
isometric quotient $E/G$ satisfies
\[
    r_f(F,E/G)\le\frac{\pi}{4}.
\]
Applying this to the Hopf fibrations gives

\begin{theorem}[Theorem \ref{thm:projective-spaces} and Theorem \ref{thm:projective-rigidity}]
\label{thm:projective-intro}
Let $P=\CP^n,\HP^n$ with the Hopf normalization.  Then every
closed connected two-sided immersed hypersurface $F:\Sigma\to P$ satisfies
\[
    r_f(F,P)\le \frac{\pi}{4}.
\]
Equality implies that $F$ is locally congruent to a
projective Clifford hypersurface
\[
    \left\{[u:v]\in\mathbb P_{\K}(\K^{k+1}\oplus\K^{n-k}):
    \|u\|=\|v\|\right\},
    \qquad
    \K=\CC \text{ or } \HH.
\]
\end{theorem}

\subsection{Strategy of proof} \label{subsec:strategy}  The estimate follows from two ingredients. First, the focal-radius assumption, the normal index form, and $\Ric_p\ge p$ give sharp bounds for suitable $p$-traces of the second fundamental form. Second, the Gauss equation and the explicit formula for $q^{(p)}(A\owedge A)$ convert these bounds into nonnegativity of $q_\Sigma^{(p)}$. A nonzero harmonic $p$-form then forces equality in all inequalities.  

In the equality case, one first obtains the Clifford-type splitting of the pullback normal cylinder. Completeness lets the normal exponential map extend smoothly to the two focal points. Comparing the first nonzero Jacobi term at each focal point forces the collapsed factors to be round, so the lifted hypersurface is locally congruent to the standard Clifford hypersurface. If $M$ is compact, every point lies within distance $\pi/4$ of the hypersurface; the open Clifford collar is dense, and continuity of curvature gives $\sec_M\equiv1$. Thus $M$ is a spherical space form.  

The paper is organized as follows. Section \ref{sec:curvature-assumptions} records curvature conditions implying the two abstract assumptions. Section \ref{sec:general-p-form-estimate} proves Theorem \ref{thm:gp-focal-estimate}. Section \ref{sec:equality-rigidity} proves Theorem \ref{thm:gp-rigidity}. Section \ref{sec:projective-spaces} proves the quotient-lifting principle and applies it to projective spaces. 
\section{Examples of the curvature assumptions}
\label{sec:curvature-assumptions}

We record several standard curvature conditions which imply Assumptions
\ref{ass:gp-ambient-p-Bochner} and \ref{ass:gp-Ric-p}. The first one is a
convenient eigenvalue criterion for the Weitzenb\"ock term.

\begin{lemma}
\label{lem:gp-eigenvalue-sufficient-condition}
Fix $x\in M$ and an $m$-plane $W\subset T_xM$. Let
$\lambda_1(W)\le\cdots\le\lambda_{\binom m2}(W)$ be the eigenvalues of the
curvature operator of $R^M|_W$ on $\Lambda^2W$. If
\[
    \lambda_1(W)+\cdots+\lambda_{m-p}(W)\ge m-p
\]
for every $x\in M$ and every such $W$, then Assumption
\ref{ass:gp-ambient-p-Bochner} holds. This is the form of the
Petersen--Wink eigenvalue condition needed below \cite{PetersenWink2020}.
\end{lemma}
\begin{proof}
Let $\Omega_\alpha$ be an orthonormal eigenbasis of $\Lambda^2W$, with
eigenvalues $\lambda_\alpha(W)$. On $\Lambda^pW^*$, write
$L_{ij}:=\varepsilon_i\iota_j-\varepsilon_j\iota_i$, and
$L_\Omega:=\sum_{i<j}\Omega_{ij}L_{ij}$ if
$\Omega=\sum_{i<j}\Omega_{ij}e^i\wedge e^j$. Then the Weitzenb\"ock curvature term on $\Lambda^pW^*$ is
\[
    q^{(p)}(R^M|_W)
    =-\sum_{i,j} e^i\wedge \iota_{e_j}\bigl(R(e_i,e_j)\cdot\bigr)=
    -\sum_\alpha \lambda_\alpha(W)L_{\Omega_\alpha}^2 .
\]
Fix $|\eta|=1$ and put
$a_\alpha:=\langle -L_{\Omega_\alpha}^2\eta,\eta\rangle
=|L_{\Omega_\alpha}\eta|^2$. For the standard basis $e^i\wedge e^j$, write
$\eta=\sum_{|I|=p}\eta_Ie^I$. Then
\begin{align*}
    |L_{ij}\eta|^2
    &=
    \sum_{\#(\{i,j\}\cap I)=1}|\eta_I|^2
    \le 1,                                                     \\
    \sum_{i<j}|L_{ij}\eta|^2
    &=
    \sum_{i<j}
    \sum_{\#(\{i,j\}\cap I)=1}
    |\eta_I|^2                                                 \\
    &=
    \sum_{|I|=p}
    \#\{(i,j):i<j,\ \#(\{i,j\}\cap I)=1\}
    |\eta_I|^2                                                 \\
    &=
    \sum_{|I|=p}|I||I^c|
    |\eta_I|^2                                                 \\
    &=
    p(m-p).
\end{align*}
The same identities hold for any orthonormal basis $\Omega_\alpha$ of
$\Lambda^2W$. Hence $0\le a_\alpha/p\le1$ and
$\sum_\alpha a_\alpha/p=m-p$. Since the eigenvalues are ordered increasingly, we have
\[
    \sum_\alpha \lambda_\alpha(W)\frac{a_\alpha}{p}
    \ge
    \lambda_1(W)+\cdots+\lambda_{m-p}(W).
\]
Therefore
\begin{align*}
    \langle q^{(p)}(R^M|_W)\eta,\eta\rangle
    &=
    \sum_\alpha \lambda_\alpha(W)a_\alpha                                      \\
    &\ge
    p\bigl(\lambda_1(W)+\cdots+\lambda_{m-p}(W)\bigr)                           \\
    &\ge
    p(m-p).
\end{align*}
This proves Assumption \ref{ass:gp-ambient-p-Bochner}.
\end{proof}
\begin{example}
\label{ex:round-space-forms}
Let $M=S^{m+1}(1)$, or more generally $M=S^{m+1}(1)/\Gamma$. Then for every
$m$-plane $W\subset T_xM$,
\[
    q^{(p)}(R^M|_W)
    =
    p(m-p)\Id_{\Lambda^pW^*},
    \qquad
    \Ric_p=p.
\]
Thus Assumptions \ref{ass:gp-ambient-p-Bochner} and \ref{ass:gp-Ric-p} hold
with equality.
\end{example}

\begin{example}
\label{ex:curvature-operator-one}
Assume $\mathcal R_M\ge \Id_{\Lambda^2T_xM}$. Then the restriction to every
$m$-plane $W\subset T_xM$ satisfies
\[
    \mathcal R_M|_{\Lambda^2W}\ge \Id_{\Lambda^2W}.
\]
Hence Lemma \ref{lem:gp-eigenvalue-sufficient-condition} gives
\[
    q^{(p)}(R^M|_W)
    \ge
    p(m-p)\Id_{\Lambda^pW^*}.
\]
Moreover, for every unit vector $X$ and every orthonormal $p$-frame
$E_1,\dots,E_p\perp X$,
\[
    \sum_{a=1}^p R(E_a,X,E_a,X)\ge p.
\]
Thus Assumptions \ref{ass:gp-ambient-p-Bochner} and \ref{ass:gp-Ric-p} hold.
\end{example}

\begin{example}[The case $p=1$]
\label{ex:p-one-sec}
When $p=1$, the curvature assumptions follow from $\sec_M\ge1$. Indeed,
let $W\subset T_xM$ be an $m$-plane and let $v\in W$ be a unit vector. Choose
an orthonormal basis $e_1,\ldots,e_m$ of $W$ with $e_1=v$. We compute
\begin{align*}
    \left\langle q^{(1)}(R^M|_W)v^\flat,v^\flat\right\rangle
    &=
    -\sum_{i,j=1}^{m}
    \left\langle
    e^i\wedge\iota_{e_j}\bigl(R(e_i,e_j)e^1\bigr),
    e^1
    \right\rangle                                      \\
    &=
    -\sum_{j=1}^{m}
    \bigl(R(e_1,e_j)e^1\bigr)(e_j)                    \\
    &=
    \sum_{j=2}^{m}R(e_j,v,e_j,v)                      \\
    &\ge m-1.
\end{align*}
Hence
\[
    q^{(1)}(R^M|_W)\ge (m-1)\Id_{W^*}.
\]
Also $\Ric_1\ge1$ follows directly from $\sec_M\ge1$. Thus Assumptions
\ref{ass:gp-ambient-p-Bochner} and \ref{ass:gp-Ric-p} hold for $p=1$.
\end{example}

The following elementary consequence of $\mathrm{PIC1}$ is the input needed for the two-form case.  This condition is a variant of the isotropic-curvature positivity studied by Micallef--Wang \cite{MicallefWang1993}.
\begin{lemma}
\label{lem:PIC1-q2-nonnegative}
Let $(V,g)$ be an inner-product space with $\dim V\ge4$, and let
$S\in\mathcal C_{\mathrm{PIC1}}(V)$. Then, for every subspace
$W\subset V$ of dimension at least four,
\[
    q^{(2)}(S|_W)\ge0
    \qquad\text{on }\Lambda^2W^*.
\]
\end{lemma}

\begin{proof}
It suffices to prove the claim for an arbitrary $2$-form
$\omega\in\Lambda^2W^*$. Choose an orthonormal basis of $W$ such that
\[
    \omega
    =
    \sum_{\alpha=1}^{\ell}
    \mu_\alpha e^{2\alpha-1}\wedge e^{2\alpha},
    \qquad
    2\ell\le m.
\]
The Weitzenb\"ock curvature term on $2$-forms gives
\begin{align*}
    \left\langle q^{(2)}(S|_W)\omega,\omega\right\rangle
    &=
    \sum_{\alpha=1}^{\ell}\mu_\alpha^2
    \sum_{k\neq 2\alpha-1,2\alpha}
    \left(
        S_{2\alpha-1,k,2\alpha-1,k}
        +
        S_{2\alpha,k,2\alpha,k}
    \right)                                                    \\
    &\quad
    -
    4\sum_{\alpha<\beta}
    \mu_\alpha\mu_\beta
    S_{2\alpha-1,2\alpha,2\beta-1,2\beta}.
\end{align*}
For $\alpha<\beta$, set
\begin{align*}
    A_{\alpha\beta}
    &:={}
    S_{2\alpha-1,2\beta-1,2\alpha-1,2\beta-1}
    +
    S_{2\alpha-1,2\beta,2\alpha-1,2\beta}                      \\
    &\quad
    +
    S_{2\alpha,2\beta-1,2\alpha,2\beta-1}
    +
    S_{2\alpha,2\beta,2\alpha,2\beta},
    \\
    B_{\alpha\beta}
    &:={}
    S_{2\alpha-1,2\alpha,2\beta-1,2\beta}.
\end{align*}
Therefore
\begin{align*}
    \left\langle q^{(2)}(S|_W)\omega,\omega\right\rangle
    &=
    \sum_{\alpha<\beta}
    \left((\mu_\alpha^2+\mu_\beta^2)A_{\alpha\beta}
    -
    4\mu_\alpha\mu_\beta B_{\alpha\beta} \right)       \\
    &\quad
    +
    \sum_{\alpha=1}^{\ell}\mu_\alpha^2
    \sum_{k>2\ell}
    \left(
        S_{2\alpha-1,k,2\alpha-1,k}
        +
        S_{2\alpha,k,2\alpha,k}
    \right).
\end{align*}
The $\lambda=1$ case of PIC1 \eqref{eq:PIC1-cone}, applied to the four-frame
$(e_{2\alpha-1},e_{2\alpha},e_{2\beta-1},e_{2\beta})$ and then with
$e_{2\beta}$ replaced by $-e_{2\beta}$, gives
\[
    A_{\alpha\beta}-2B_{\alpha\beta}\ge0,
    \qquad
    A_{\alpha\beta}+2B_{\alpha\beta}\ge0.
\]
Thus $A_{\alpha\beta}\ge2|B_{\alpha\beta}|$, and hence
\begin{align*}
 (\mu_\alpha^2+\mu_\beta^2)A_{\alpha\beta}
    -4\mu_\alpha\mu_\beta B_{\alpha\beta}
&\ge
    (\mu_\alpha^2+\mu_\beta^2)A_{\alpha\beta}
    -4|\mu_\alpha\mu_\beta|\,|B_{\alpha\beta}| \\ 
&\ge
    A_{\alpha\beta}
    (|\mu_\alpha|-|\mu_\beta|)^2
    \ge0.
\end{align*}
The terms involving indices outside the support of the normal form of
$\omega$ are also nonnegative: by the $\lambda=0$ case of PIC1 \eqref{eq:PIC1-cone},
\[
    S_{2\alpha-1,k,2\alpha-1,k}
    +
    S_{2\alpha,k,2\alpha,k}
    \ge0.
\]
Combining these inequalities gives
$\left\langle q^{(2)}(S|_W)\omega,\omega\right\rangle\ge0$.
\end{proof}

\begin{example}[The case $p=2$]
\label{ex:p-two-PIC1}
When $p=2$, normalized $\mathrm{PIC1}\ge1$ implies Assumptions
\ref{ass:gp-ambient-p-Bochner} and \ref{ass:gp-Ric-p}. We use the normalization
\[
    R^M-g_M\owedge g_M\in\mathcal C_{\mathrm{PIC1}}(T_xM)
    \qquad\text{for every }x\in M,
\]
where $g_M\owedge g_M$ is the curvature tensor of constant sectional
curvature $1$. Set $S:=R^M-g_M\owedge g_M$.

By Lemma \ref{lem:PIC1-q2-nonnegative}, $q^{(2)}(S|_W)\ge0$ for every
$m$-plane $W\subset T_xM$. Since the constant-curvature-one tensor satisfies
\[
    q^{(2)}(g_M\owedge g_M|_W)
    =
    2(m-2)\Id_{\Lambda^2W^*},
\]
we obtain
\[
    q^{(2)}(R^M|_W)
    =
    q^{(2)}(g_M\owedge g_M|_W)+q^{(2)}(S|_W)
    \ge
    2(m-2)\Id_{\Lambda^2W^*}.
\]
This proves Assumption \ref{ass:gp-ambient-p-Bochner} for $p=2$.

Similarly, the $\lambda=0$ case of PIC1 gives, for every unit vector $X$ and
every orthonormal two-frame $E_1,E_2\perp X$,
\[
    S(E_1,X,E_1,X)+S(E_2,X,E_2,X)\ge0.
\]
Therefore
\[
    R^M(E_1,X,E_1,X)+R^M(E_2,X,E_2,X)
    =
    2+S(E_1,X,E_1,X)+S(E_2,X,E_2,X)
    \ge2.
\]
Thus Assumption \ref{ass:gp-Ric-p} holds for $p=2$.
\end{example}

\bigskip

\section{Proof of the focal radius estimate}
\label{sec:general-p-form-estimate}

\begin{lemma}
\label{lem:gp-index-nonnegative}
For $\eta=\pm\nu$, let $\gamma(t)=\exp_x(t\eta)$ and $A_\eta$ be the shape operator of $\Sigma$ with respect to $\eta$. If there is no focal point
along $\gamma$ on $[0,L]$, then for every field $V(t)\perp\gamma'(t)$ with
$V(0)\in T_x\Sigma$ and $V(L)=0$, the index form satisfies
\begin{equation}\label{eq:gp-index-form}
    I_\eta(V,V)
    =
    \int_0^L
    \left(
        |D_tV|^2
        -
        R(V,\gamma',V,\gamma')
    \right)\,dt
    +
    A_\eta(V(0),V(0))\ge 0 .
\end{equation}
Moreover, equality holds if and only if $V$ is a $\Sigma$-Jacobi field, namely
\[
    D_t^2V+R(V,\gamma')\gamma'=0,
    \qquad
    D_tV(0)=A_\eta V(0),
    \qquad
    V(L)=0.
\]
\end{lemma}

\begin{proof}
Let $Y(t):T_x\Sigma\to\gamma'(t)^\perp$ be the Jacobi tensor with
$Y(0)=\Id$ and $Y'(0)=A_\eta$. Since there is no focal point on $[0,L]$,
$Y(t)$ is invertible. It satisfies
\[
    Y''+R_\gamma Y=0,
    \qquad
    R_\gamma Z=R(Z,\gamma')\gamma'.
\]
This implies that the quantity $Y'^*Y-Y^*Y'$ is constant and vanishes at $t=0$, hence
$Y'^*Y=Y^*Y'$.

Write $V=YW$. Since $V(L)=0$, we have $W(L)=0$. Using the Jacobi equation and
$Y'^*Y=Y^*Y'$, one computes
\begin{align*}
    |D_tV|^2-\langle R_\gamma V,V\rangle
    &=
    |Y'W+YW'|^2-\langle R_\gamma YW,YW\rangle    \\
    &=
    |YW'|^2
    +2\langle Y'W,YW'\rangle
    +|Y'W|^2
    +\langle Y''W,YW\rangle  \\
    &=
    |YW'|^2
    +
    \frac{d}{dt}\langle Y'W,YW\rangle .
\end{align*}
Integration by parts gives
\begin{align*}
    I_\eta(V,V)
    &= \int_0^L |YW'|^2\,dt + \left[\langle Y'W,YW\rangle\right]_0^L + A_\eta(V(0),V(0)) \\
    &= \int_0^L |YW'|^2\,dt - \langle Y'(0)W(0),Y(0)W(0)\rangle + A_\eta(V(0),V(0)) \\
    &= \int_0^L |YW'|^2\,dt \ge0.
\end{align*}
\end{proof}

\begin{lemma}
\label{lem:gp-index-trace}
Assume $\Ric_p\ge p$. If the two-sided focal radius of $\Sigma$ is at least
$L$, then for every $r\ge p$ and every $r$-plane $P\subset T_x\Sigma$,
\begin{equation}\label{eq:gp-trace-A-bound}
    |\tr_P A|\le r\cot L.
\end{equation}
\end{lemma}

\begin{proof}
First, $\Ric_p\ge p$ implies $\Ric_r\ge r$ for every $r\ge p$. Let $X$ be
unit, let $E_1,\dots,E_r\perp X$ be orthonormal, and set
$K_a=R(E_a,X,E_a,X)$. Summing $\sum_{a\in J}K_a\ge p$ over all $p$-element
subsets $J\subset\{1,\dots,r\}$ gives $\sum_{a=1}^rK_a\ge r$.

Fix $\eta=\pm\nu$. Let $E_1,\dots,E_r$ be an orthonormal basis of $P$,
parallel transported along $\gamma(t)=\exp_x(t\eta)$. Set
$f(t)=\sin(L-t)/\sin L$. Then $f(0)=1$, $f(L)=0$, $f''+f=0$, and
$f'(0)=-\cot L$. Putting $V_a=fE_a$ in Lemma
\ref{lem:gp-index-nonnegative} gives
\[
\begin{aligned}
0
&\le \sum_{a=1}^r I_\eta(V_a,V_a)\\
&=
\int_0^L
\left(
    r(f')^2
    -
    f^2\sum_{a=1}^r R(E_a,\gamma',E_a,\gamma')
\right)\,dt
+
\tr_P A_\eta\\
&\le
r\int_0^L((f')^2-f^2)\,dt+\tr_P A_\eta .
\end{aligned}
\]
Since $f''=-f$,
\[
    \int_0^L((f')^2-f^2)\,dt
    =
    [ff']_0^L
    =
    \cot L.
\]
Thus $0\le r\cot L+\tr_P A_\eta$. Taking $\eta=\nu$ gives
$\tr_P A\ge-r\cot L$, while taking $\eta=-\nu$ gives
$\tr_P A\le r\cot L$. Hence \eqref{eq:gp-trace-A-bound} follows.
\end{proof}

\begin{lemma}
\label{lem:gp-qT}
Under Assumption \ref{ass:gp-ambient-p-Bochner},
\[
    q_T^{M,p}\ge p(m-p)
\]
on $\Lambda^pT^*\Sigma$.
\end{lemma}

\begin{proof}
Apply Assumption \ref{ass:gp-ambient-p-Bochner} with $W=T_x\Sigma$.
\end{proof}

\begin{lemma}
\label{lem:gp-qA-formula}
Let $\{e_1,\ldots,e_m\}$ be a principal frame on $\Sigma$ with
$Ae_i=\kappa_i e_i$. For a multi-index
$I=\{i_1<\cdots<i_p\}\subset\{1,\ldots,m\}$, write
\[
    e^I=e^{i_1}\wedge\cdots\wedge e^{i_p},
    \qquad
    s_I=\sum_{i\in I}\kappa_i.
\]
If $\omega=\sum_{|I|=p}\omega_I e^I\in\Omega^p(\Sigma)$, then
\begin{equation}\label{eq:gp-qA-formula}
    \langle q_A^{(p)}\omega,\omega\rangle
    =
    \sum_{|I|=p}s_I(H-s_I)|\omega_I|^2.
\end{equation}
\end{lemma}

\begin{proof}
Define the actions $\varepsilon_i:\Omega^p(\Sigma)\to\Omega^{p+1}(\Sigma)$ and $\iota_i:\Omega^p(\Sigma) \to \Omega^{p-1}(\Sigma)$ by
\[  \varepsilon_i\omega=e^i\wedge\omega,\quad \iota_i=\iota_{e_i}.\]
With this we have
\begin{align*}
    q^{(p)}_R\omega = -\sum_{i,j}e^i\wedge i_{e_j}(R(e_i,e_j)\omega) = -\sum_{i,j}\varepsilon_i\iota_j(R(e_i,e_j)\omega).
\end{align*}
For $a\neq b$, define $L_{ab}:\Omega^p(\Sigma)\to\Omega^p(\Sigma)$ by $L_{ab}=\varepsilon_a\iota_b - \varepsilon_b\iota_a$. In the principal frame,
$A\owedge A$ is diagonal on $\Lambda^2T\Sigma$ with sectional entries
$\kappa_a\kappa_b$. Therefore
\[
    q_A^{(p)}
    =
    -\sum_{a<b}\kappa_a\kappa_b L_{ab}^2.
\]
For a basis $p$-form $e^I$,
\[
    -L_{ab}^2e^I
    =
    \begin{cases}
        e^I, & \#(\{a,b\}\cap I)=1,\\
        0, & \text{otherwise}.
    \end{cases}
\]
Thus
\[
    q_A^{(p)}e^I
    =
    \sum_{a\in I,\ b\notin I}\kappa_a\kappa_b\,e^I
    =
    s_I(H-s_I)e^I.
\]
Since the basis $e^I$ diagonalizes $q_A^{(p)}$, the formula follows.
\end{proof}

\begin{lemma}
\label{lem:gp-qA}
If the two-sided focal radius of $\Sigma$ is at least $L$, then
\[
    q_A^{(p)}\ge -p(m-p)\cot^2L.
\]
\end{lemma}

\begin{proof}
For $|I|=p$, Lemma \ref{lem:gp-index-trace} gives
\[
    |s_I|\le p\cot L,
    \qquad
    |H-s_I|\le (m-p)\cot L.
\]
Hence $s_I(H-s_I)\ge -p(m-p)\cot^2L$. The result follows from Lemma
\ref{lem:gp-qA-formula}.
\end{proof}

\begin{proposition}
\label{prop:gp-Bochner}
If the two-sided focal radius of $\Sigma$ is at least $L$, then
\begin{equation}\label{eq:gp-combined-Bochner}
    q_\Sigma^{(p)}
    \ge
    p(m-p)(1-\cot^2L).
\end{equation}
\end{proposition}

\begin{proof}
By \eqref{eq:gp-Gauss-split}, $q_\Sigma^{(p)}=q_T^{M,p}+q_A^{(p)}$. The
estimate follows from Lemma \ref{lem:gp-qT} and Lemma \ref{lem:gp-qA}.
\end{proof}

\bigskip

\subsection{Proof of Theorem \ref{thm:gp-focal-estimate}}
Since $b_p(\Sigma)\neq0$, choose nonzero harmonic $p$-form $\omega$ on $\Sigma$. The
integrated Bochner identity gives
\[
    0
    =
    \int_\Sigma |\nabla\omega|^2
    +
    \int_\Sigma \langle q_\Sigma^{(p)}\omega,\omega\rangle.
\]
If the two-sided focal radius is at least $L$, Proposition
\ref{prop:gp-Bochner} gives
\[
    0
    \ge
    \int_\Sigma |\nabla\omega|^2
    +
    p(m-p)(1-\cot^2L)\int_\Sigma|\omega|^2.
\]
If \(r_f(F,M)>\pi/4\), choose \(\frac{\pi}{4}<L<\min\{r_f(F,M),\pi/2\}\). Then $1-\cot^2L>0$, contradicting $\omega\neq0$. Hence
$r_f(F,M)\le\frac{\pi}{4}$.

\bigskip

\section{Proof of rigidity}
\label{sec:equality-rigidity}

The equality proof has two parts.  On the open regular collar we first show
that equality in the Bochner and index-form inequalities forces a two-block
splitting of the shape operator and a doubly warped product metric.  We then
use completeness to extend the normal exponential map to the two limiting
focal hypersurfaces; the Jacobi fields at focal points identify the two collapsed
factors as round spheres.

Throughout this section we assume that the hypotheses of Theorem
\ref{thm:gp-rigidity} hold.  In particular, $M$ is complete, $\Sigma$ is
connected, and
\[
    r_f(F,M)=\frac{\pi}{4},
    \qquad
    1\le p\le \frac{m}{2}.
\]
Let
\[
    \Phi:\Sigma\times\left(-\frac{\pi}{4},\frac{\pi}{4}\right)\to M
\]
be the normal exponential map, set $g=\Phi^*g_M$, and write
\[
    \Sigma_t=\Sigma\times\{t\},
    \qquad
    g_t=g|_{\Sigma_t}.
\]
The vector field $\nu=\partial_t$ is unit and orthogonal to every slice
$\Sigma_t$. Let $A(t)$ be the shape operator of $\Sigma_t$ in
$(\Sigma\times(-\pi/4,\pi/4),g)$ with respect to $\nu$.

Fix $0\neq[\omega]\in H^p(\Sigma;\mathbb R)$. Using the identification
$\Sigma_t\simeq\Sigma$, we regard $[\omega]$ as a cohomology class on
$\Sigma_t$. Let $\omega_t$ be its harmonic representative with respect to
$g_t$. Since $\Sigma_t$ is naturally identified with $\Sigma$, this cohomology class
is nonzero for every $t$.  Once the Bochner equality below shows that
$\omega_t$ is parallel, connectedness of $\Sigma$ implies that $\omega_t$ is
nowhere vanishing.

\subsection{Bochner equality on the parallel hypersurfaces}
\label{subsec:gp-equality-parallel-hypersurfaces}

We first record the Riccati equation for the shape operator along the normal
flow.

\begin{lemma}
\label{lem:gp-shape-evolution}
Let $A(t)$ be the shape operator of $\Sigma_t$, and define the radial curvature operator $B_\nu$ by
\[
    g(B_\nu X,Y)=R(X,\nu,Y,\nu).
\]
Then
\[
    \nabla_\nu A+A^2+B_\nu=0.
\]
\end{lemma}

\begin{proof}
Extend $X\in T\Sigma_t$ by $[\nu,X]=0$. Since $\nabla_\nu\nu=0$, we have
$\nabla_\nu X=\nabla_X\nu=AX$. Hence
\[
    (\nabla_\nu A)X
    =
    \nabla_\nu(AX)-A(\nabla_\nu X)
    =
    \nabla_\nu\nabla_X\nu-A^2X.
\]
By the curvature convention,
\[
    R(\nu,X)\nu=\nabla_\nu\nabla_X\nu,
    \qquad
    g(R(\nu,X)\nu,Y)=-R(X,\nu,Y,\nu)=-g(B_\nu X,Y).
\]
Thus $R(\nu,X)\nu=-B_\nu X$, and the assertion follows.
\end{proof}

The next step is to propagate the equality case of the Bochner argument to each
parallel hypersurface $\Sigma_t$.

\begin{lemma}
\label{lem:gp-leafwise-equality}
Let $\{e_1,\ldots,e_m\}$ be a principal frame on $\Sigma_t$, with
$A(t)e_i=\kappa_i e_i$. Write
\[
    \omega_t=\sum_{|I|=p}\omega_I e^I,
    \qquad
    s_I=\sum_{i\in I}\kappa_i.
\]
Then for every $t\in(-\pi/4,\pi/4)$, the following hold:
\begin{enumerate}
    \item[(i)] The harmonic form $\omega_t\in\mathcal H^p(\Sigma_t)$
    satisfies
    \[
        \nabla^{\Sigma_t}\omega_t=0,
        \qquad
        q_{\Sigma_t}^{(p)}\omega_t=0.
    \]

    \item[(ii)] For every $I\subset\{1,\ldots,m\}$ with $|I|=p$ and
    $\omega_I\neq0$, one has either
    \[
        s_I=-p\cot\left(\frac{\pi}{4}-t\right),
        \qquad
        H-s_I=(m-p)\cot\left(\frac{\pi}{4}+t\right),
    \]
    or
    \[
        s_I=p\cot\left(\frac{\pi}{4}+t\right),
        \qquad
        H-s_I=-(m-p)\cot\left(\frac{\pi}{4}-t\right).
    \]
\end{enumerate}
\end{lemma}

\begin{proof}
For every $J\subset\{1,\ldots,m\}$ with $|J|=r\ge p$, the proof of Lemma
\ref{lem:gp-index-trace} applied to $\Sigma_t$ gives
\[
    -r\cot\left(\frac{\pi}{4}-t\right)
    \le
    \sum_{j\in J}\kappa_j
    \le
    r\cot\left(\frac{\pi}{4}+t\right),
\]
because the remaining focal distances from $\Sigma_t$ are
$\pi/4-t$ and $\pi/4+t$. As in the proof of Proposition
\ref{prop:gp-Bochner}, using
\[
    \cot\left(\frac{\pi}{4}-t\right)
    \cot\left(\frac{\pi}{4}+t\right)=1,
\]
we get $q_{\Sigma_t}^{(p)}\ge0$. Since $\omega_t$ is harmonic,
\[
    0
    =
    \int_{\Sigma_t}|\nabla^{\Sigma_t}\omega_t|^2
    +
    \int_{\Sigma_t}\langle q_{\Sigma_t}^{(p)}\omega_t,\omega_t\rangle.
\]
Therefore $\nabla^{\Sigma_t}\omega_t=0$ and
$q_{\Sigma_t}^{(p)}\omega_t=0$.

For $|I|=p$, the trace bounds give
\[
    -p\cot\left(\frac{\pi}{4}-t\right)
    \le
    s_I
    \le
    p\cot\left(\frac{\pi}{4}+t\right),
\]
and, since $m-p\ge p$,
\[
    -(m-p)\cot\left(\frac{\pi}{4}-t\right)
    \le
    H-s_I
    \le
    (m-p)\cot\left(\frac{\pi}{4}+t\right).
\]
Thus $s_I(H-s_I)\ge -p(m-p)$. Since equality holds in the pointwise Bochner
nonnegativity, every $I$ with $\omega_I\neq0$ satisfies
$s_I(H-s_I)=-p(m-p)$. The equality cases are exactly the two opposite
corners displayed above.
\end{proof}

We now use the equality in the index estimate to upgrade trace
equalities into operator equalities.

\begin{lemma}
\label{lem:gp-focalpoint-index-equality}
Let $P\subset T_x\Sigma_t$ be an $r$-plane with $r\ge p$.
\begin{enumerate}
    \item[(i)] If $\tr_P A=-r\cot(\frac{\pi}{4}-t)$, then
    \[
        A|_P=-\cot\left(\frac{\pi}{4}-t\right)\Id,
        \qquad
        \tr_PB_\nu=r,
        \qquad
        B_\nu P\subset P.
    \]

    \item[(ii)] If $\tr_P A=r\cot(\frac{\pi}{4}+t)$, then
    \[
        A|_P=\cot\left(\frac{\pi}{4}+t\right)\Id,
        \qquad
        \tr_PB_\nu=r,
        \qquad
        B_\nu P\subset P.
    \]
\end{enumerate}
\end{lemma}

\begin{proof}
We prove (i). The proof of (ii) follows by applying the same argument to the
normal $-\nu$.

Set $L=\pi/4-t$. Let $P(s)$ denote the parallel translate of $P$ along the geodesic. Let $u_1(s),\ldots,u_r(s)$ be an orthonormal basis of $P(s)$,
parallel transported along the forward normal geodesic, and set
\[
    f(s)=\frac{\sin(L-s)}{\sin L},
    \qquad
    V_a(s)=f(s)u_a(s),
    \qquad
    0\le s\le L.
\]
Then $f(0)=1$, $f(L)=0$, $f''+f=0$, and $f'(0)=-\cot L$.  For each $a$,
\[
    I_\nu(V_a,V_a)
    =
    \int_0^L
    \left(
        (f')^2
        -
        f^2R(u_a,\nu,u_a,\nu)
    \right)\,ds
    +
    A(u_a,u_a).
\]
Summing over $a$, we get
\[
\begin{aligned}
\sum_{a=1}^r I_\nu(V_a,V_a)
&=
r\int_0^L(f')^2\,ds
-
\int_0^L f^2\tr_{P(s)}B_\nu\,ds
+
\tr_P A\\
&=
r\int_0^L((f')^2-f^2)\,ds
-
\int_0^L f^2(\tr_{P(s)}B_\nu-r)\,ds
+\tr_P A.
\end{aligned}
\]
Since $f''=-f$,
\[
    \int_0^L((f')^2-f^2)\,ds
    =
    [ff']_0^L
    =
    \cot L.
\]
Using $\tr_PA=-r\cot L$, we obtain
\[
    \sum_{a=1}^r I_\nu(V_a,V_a)
    =
    -
    \int_0^L f^2(\tr_{P(s)}B_\nu-r)\,ds.
\]
The lower bound $\Ric_r\ge r$ gives $\tr_{P(s)}B_\nu-r\ge0$ for
$0\le s<L$.  On the other hand, Lemma \ref{lem:gp-index-nonnegative} gives
$I_\nu(V_a,V_a)\ge0$ for each $a$.  Hence both sides must vanish:
\[
    I_\nu(V_a,V_a)=0
    \quad\text{for every }a,
    \qquad
    \tr_{P(s)}B_\nu=r
    \quad\text{for }0\le s<L.
\]
By the equality statement in Lemma \ref{lem:gp-index-nonnegative}, each
$V_a$ is a $\Sigma_t$-Jacobi field.  Therefore
\[
    A u_a=D_sV_a(0)=f'(0)u_a=-\cot L\,u_a.
\]
This proves $A|_P=-\cot L\,\Id$.

The identity $\tr_{P(s)}B_\nu=r$ at $s=0$ gives $\tr_PB_\nu=r$.
Since every $r$-plane $\Pi\subset T\Sigma_t$ satisfies $\tr_\Pi B_\nu\ge r$,
the plane $P$ is a minimizer of the function \(\Pi\longmapsto \tr_\Pi B_\nu\) on the Grassmannian.  Taking the first variation of this function by replacing
one unit vector $u\in P$ with $u+\varepsilon U$, $U\in P^\perp$, gives \(g(B_\nu u,U)=0\). Hence $B_\nu P\subset P$.
\end{proof}

\subsection{Splitting in the equality case}
\label{subsec:gp-equality-splitting}

\begin{lemma}
\label{lem:gp-pointwise-equality-blocks}
At every point of $\Sigma_t$, there is an orthogonal splitting
\[
    T\Sigma_t=E_t\oplus F_t
\]
such that
\[
    A|_{E_t}
    =
    -\cot\left(\frac{\pi}{4}-t\right)\Id,
    \qquad
    A|_{F_t}
    =
    \cot\left(\frac{\pi}{4}+t\right)\Id.
\]
Moreover,
\[
    \tr_{E_t}B_\nu=\dim E_t,
    \qquad
    \tr_{F_t}B_\nu=\dim F_t,
    \qquad
    \{\dim E_t,\dim F_t\}=\{p,m-p\}.
\]
\end{lemma}

\begin{proof}
At every point of $\Sigma_t$, choose a principal frame
$\{e_1,\ldots,e_m\}$ for $A$. Since $\omega_t\neq0$, there is
$I\subset\{1,\ldots,m\}$, $|I|=p$, such that $\omega_I\neq0$.

If the first alternative in Lemma \ref{lem:gp-leafwise-equality}(ii) holds, set
\[
    E_t=\operatorname{span}\{e_i:i\in I\},
    \qquad
    F_t=E_t^\perp .
\]
Then
\[
    \tr_{E_t}A
    =
    -p\cot\left(\frac{\pi}{4}-t\right),
    \qquad
    \tr_{F_t}A
    =
    (m-p)\cot\left(\frac{\pi}{4}+t\right).
\]
Applying Lemma \ref{lem:gp-focalpoint-index-equality}(i) to $E_t$ and Lemma
\ref{lem:gp-focalpoint-index-equality}(ii) to $F_t$, we get
\[
    A|_{E_t}
    =
    -\cot\left(\frac{\pi}{4}-t\right)\Id,
    \qquad
    A|_{F_t}
    =
    \cot\left(\frac{\pi}{4}+t\right)\Id,
\]
and
\[
    \tr_{E_t}B_\nu=\dim E_t,
    \qquad
    \tr_{F_t}B_\nu=\dim F_t.
\]

If the second alternative in Lemma \ref{lem:gp-leafwise-equality}(ii) holds, set
\[
    E_t=\operatorname{span}\{e_j:j\in I^c\},
    \qquad
    F_t=E_t^\perp .
\]
Then
\[
    \tr_{E_t}A
    =
    -(m-p)\cot\left(\frac{\pi}{4}-t\right),
    \qquad
    \tr_{F_t}A
    =
    p\cot\left(\frac{\pi}{4}+t\right).
\]
Again Lemma \ref{lem:gp-focalpoint-index-equality}(i) applied to $E_t$ and
Lemma \ref{lem:gp-focalpoint-index-equality}(ii) applied to $F_t$ give the
same conclusions. In either case,
\[
    \{\dim E_t,\dim F_t\}=\{p,m-p\}.
\]
\end{proof}

We next fix the normal orientation so that the lower block has dimension $p$.

\begin{lemma}
\label{lem:gp-orientation-choice}
After choosing the normal orientation on each connected component of the
regular normal-flow region, the splitting in Lemma
\ref{lem:gp-pointwise-equality-blocks} can be arranged so that
\[
    \dim E_t=p,
    \qquad
    \dim F_t=m-p.
\]
\end{lemma}

\begin{proof}
If $m=2p$, there is nothing to prove. Assume $p<m-p$.

At a point, choose a principal frame and choose $I$ with $\omega_I\neq0$.
If the first alternative in Lemma \ref{lem:gp-leafwise-equality}(ii) holds,
then Lemma \ref{lem:gp-pointwise-equality-blocks} gives
\[
    E_t=\operatorname{span}\{e_i:i\in I\},
    \qquad
    \dim E_t=p,
\]
and
\[
    H
    =
    -p\cot\left(\frac{\pi}{4}-t\right)
    +(m-p)\cot\left(\frac{\pi}{4}+t\right).
\]
If the second alternative in Lemma \ref{lem:gp-leafwise-equality}(ii) holds,
then
\[
    E_t=\operatorname{span}\{e_j:j\in I^c\},
    \qquad
    \dim E_t=m-p,
\]
and
\[
    H
    =
    p\cot\left(\frac{\pi}{4}+t\right)
    -(m-p)\cot\left(\frac{\pi}{4}-t\right).
\]
The difference between these two possible values of $H$ is
\[
    (m-2p)
    \left[
    \cot\left(\frac{\pi}{4}-t\right)
    +
    \cot\left(\frac{\pi}{4}+t\right)
    \right]
    >0.
\]
Since $H$ is continuous, the two possibilities cannot interchange on a
connected component of the regular normal-flow region. Thus, on such a
component, either $\dim E_t=p$ everywhere or $\dim E_t=m-p$ everywhere.
Reversing the normal orientation exchanges the two possibilities. Hence we may
choose the normal orientation so that $\dim E_t=p$.
\end{proof}

The next consequence is that the radial curvature operator is exactly the
identity.

\begin{lemma}
\label{lem:gp-radial-curvature-identity}
For every $t\in(-\pi/4,\pi/4)$,
\[
    B_\nu|_{T\Sigma_t}=\Id_{T\Sigma_t}.
\]
\end{lemma}

\begin{proof}
Fix a point of $\Sigma_t$. Since
$g(B_\nu X,\nu)=R^M(X,\nu,\nu,\nu)=0$ for every $X\in T\Sigma_t$, the
operator $B_\nu$ preserves $T\Sigma_t$. By Lemma
\ref{lem:gp-pointwise-equality-blocks}, we have
$\tr_{E_t}B_\nu=\dim E_t$ and $\tr_{F_t}B_\nu=\dim F_t$. Hence
$\tr_{T\Sigma_t}B_\nu=m$.

Let $\mu_1\le\cdots\le\mu_m$ be the eigenvalues of
$B_\nu|_{T\Sigma_t}$, with respect to an orthonormal eigenbasis
$v_1,\ldots,v_m$. Since $\Ric_p\ge p$, every $p$-plane
$\Pi\subset T\Sigma_t$ satisfies $\tr_\Pi B_\nu\ge p$. Applying this to
$\Pi=\operatorname{span}\{v_1,\ldots,v_p\}$ gives
\[
    \mu_1+\cdots+\mu_p\ge p.
\]
Thus $\mu_p\ge1$, for otherwise $\mu_1,\ldots,\mu_p<1$ and the last
inequality would fail. Since the eigenvalues are ordered, we get
$\mu_j\ge1$ for every $j\ge p$. Therefore
\[
    m
    =
    \tr_{T\Sigma_t}B_\nu
    =
    \sum_{j=1}^{m}\mu_j
    \ge
    p+(m-p)
    =
    m.
\]
Thus equality holds throughout. Hence $\mu_1+\cdots+\mu_p=p$ and
$\mu_{p+1}=\cdots=\mu_m=1$. Since $\mu_p\ge1$ and
$\mu_p\le\mu_{p+1}=1$, we have $\mu_p=1$. Therefore
$\mu_1,\ldots,\mu_p\le1$, while their sum is $p$, so
$\mu_1=\cdots=\mu_p=1$. Hence all eigenvalues of $B_\nu|_{T\Sigma_t}$ are
equal to $1$, and the claim follows.
\end{proof}

We also need the mixed normal curvature term in Codazzi to vanish.

\begin{lemma}
\label{lem:gp-normal-index-curvature-vanishing}
For every $t\in(-\pi/4,\pi/4)$,
\[
    R^M(X,Y,\nu,Z)=0
\]
for all $X,Y,Z\in T\Sigma_t$.
\end{lemma}

\begin{proof}
Fix $t$, a point of $\Sigma_t$, and $W\in T\Sigma_t$. For small $s$, set
\[
    N_s=\frac{\nu+sW}{|\nu+sW|}.
\]
Then $N_0=\nu$ and $N'_0=W$.

Let $\Pi\subset T\Sigma_t$ be a $p$-plane with orthonormal basis
$U_1,\ldots,U_p$. Choose smooth orthonormal vectors $U_a(s)\in N_s^\perp$
such that $U_a(0)=U_a$ and
$U_a'(0)=-\langle U_a,W\rangle\nu$. By the $\Ric_p$ inequality,
\[
    F(s):=\sum_{a=1}^p R^M(U_a(s),N_s,U_a(s),N_s)\ge p .
\]
By Lemma \ref{lem:gp-radial-curvature-identity}, $F(0)=p$, and hence
$F'(0)=0$. Differentiating at $s=0$ gives
\begin{align*}
    0
    &=
    \sum_{a=1}^p
    \Bigl[
    R^M(U_a,W,U_a,\nu)
    +
    R^M(U_a,\nu,U_a,W)
    \Bigr]                                      \\
    &=
    2\sum_{a=1}^p R^M(U_a,\nu,U_a,W),
\end{align*}
where the terms involving $U_a'(0)$ vanish because
$U_a'(0)$ is proportional to $\nu$. Thus
\[
    \sum_{a=1}^p R^M(U_a,\nu,U_a,W)=0
\]
for every $p$-plane $\Pi\subset T\Sigma_t$.

For fixed $W$, define the symmetric bilinear form
\[
    S_W(Y,Z)
    =
    \frac12
    \left(
    R^M(Y,\nu,Z,W)+R^M(Z,\nu,Y,W)
    \right)
\]
on $T\Sigma_t$. The identity above says that $\tr_\Pi S_W=0$ for every
$p$-plane $\Pi$. If $\lambda_1,\ldots,\lambda_m$ are the eigenvalues of
$S_W$, then every sum of $p$ distinct eigenvalues is zero. Since
$p\le m/2$, comparing two such sums which differ by one index shows that all
$\lambda_i$ are equal. Their $p$-fold sums are zero, so all $\lambda_i$
vanish. Thus $S_W=0$.

Therefore
\[
    R^M(Y,\nu,Z,W)=-R^M(Z,\nu,Y,W)
\]
for all $Y,Z,W\in T\Sigma_t$. Set $T(Y,Z,W)=R^M(Y,\nu,Z,W)$. The last
identity says that $T$ is skew in the first two variables, while the curvature
skew-symmetry in the last two variables gives
$T(Y,W,Z)=-T(Y,Z,W)$. Hence $T$ is alternating.

By the first Bianchi identity,
\[
    R^M(X,Y,Z,\nu)+R^M(Y,Z,X,\nu)+R^M(Z,X,Y,\nu)=0.
\]
Using pair symmetry, this becomes
\[
    T(Z,X,Y)+T(X,Y,Z)+T(Y,Z,X)=0.
\]
Since $T$ is alternating, the three terms are equal. Hence
$3T(X,Y,Z)=0$, and therefore $T=0$. Finally,
\[
    R^M(X,Y,\nu,Z)
    =
    R^M(\nu,Z,X,Y)
    =
    -R^M(Z,\nu,X,Y)
    =
    -T(Z,X,Y)
    =
    0 .
\]
This proves the claim.
\end{proof}

Now Codazzi implies that the splitting is parallel along each leaf.

\begin{lemma}
\label{lem:gp-Codazzi-leafwise-parallelness}
Let $P_t$ be the orthogonal projection onto the
$-\cot(\frac{\pi}{4}-t)$-eigenspace of $A(t)$. Then
\[
    \nabla^{\Sigma_t}P_t=0.
\]
\end{lemma}

\begin{proof}
Fix $t$, and write $\nabla=\nabla^{\Sigma_t}$. By Lemma
\ref{lem:gp-normal-index-curvature-vanishing}, the Codazzi equation reduces to
\[
    (\nabla_XA)Y=(\nabla_YA)X
\]
for all $X,Y\in T\Sigma_t$.

Set $E_t=\operatorname{im}P_t$ and $F_t=\ker P_t$. By Lemma
\ref{lem:gp-pointwise-equality-blocks},
\[
    A
    =
    -\cot\left(\frac{\pi}{4}-t\right)P_t
    +
    \cot\left(\frac{\pi}{4}+t\right)(\Id-P_t).
\]
The two coefficients are constant along $\Sigma_t$ and distinct. Hence the
Codazzi equation for $A$ gives
\[
    (\nabla_XP_t)Y=(\nabla_YP_t)X
\]
for all $X,Y\in T\Sigma_t$.

Let $K_X=\nabla_XP_t$. Since $P_t$ is an orthogonal projection, $K_X$ is
self-adjoint. Differentiating $P_t^2=P_t$ gives
\[
    K_XP_t+P_tK_X=K_X.
\]
Multiplying this identity by $P_t$ on both sides gives $P_tK_XP_t=0$.
Multiplying it by $\Id-P_t$ on both sides gives
\[
    (\Id-P_t)K_X(\Id-P_t)=0.
\]
Thus $K_X$ is off-diagonal with respect to the splitting
$T\Sigma_t=E_t\oplus F_t$:
\[
    K_X(E_t)\subset F_t,
    \qquad
    K_X(F_t)\subset E_t .
\]

We now show that all components of $K$ vanish. First take $X\in E_t$ and
$U\in F_t$. From $(\nabla_XP_t)U=(\nabla_UP_t)X$, we get $K_XU=K_UX$. The
left-hand side lies in $E_t$, while the right-hand side lies in $F_t$.
Therefore
\[
    K_XU=0,
    \qquad
    K_UX=0.
\]

Next take $X,Y\in E_t$. Since $K_X$ is off-diagonal, $K_XY\in F_t$. For every
$U\in F_t$, self-adjointness and the mixed vanishing give
\[
    \langle K_XY,U\rangle
    =
    \langle Y,K_XU\rangle
    =
    0.
\]
Hence $K_XY=0$. The same argument with the roles of $E_t$ and $F_t$ exchanged
shows that $K_UV=0$ for all $U,V\in F_t$.

Thus $K_Z=0$ for every $Z\in T\Sigma_t$, and therefore
$\nabla^{\Sigma_t}P_t=0$.
\end{proof}

The Riccati equation now propagates the splitting in the normal direction and
determines the warped product form.

\begin{lemma}
\label{lem:gp-normal-persistence-warped}
Let $P_t$ be the orthogonal projection onto the
$-\cot(\frac{\pi}{4}-t)$-eigenspace of $A(t)$. Then
\[
    \nabla_\nu P_t=0.
\]
Consequently, after the normal orientation is chosen as in Lemma
\ref{lem:gp-orientation-choice}, if $E_-$ and $E_+$ denote the
$-1$- and $+1$-eigenspaces of $A(0)$, and if $h_-^0=g_0|_{E_-}$ and
$h_+^0=g_0|_{E_+}$, then
\[
    g
    =
    dt^2
    +
    2\sin^2\left(\frac{\pi}{4}-t\right)h_-^0
    +
    2\sin^2\left(\frac{\pi}{4}+t\right)h_+^0.
\]
\end{lemma}

\begin{proof}
Since
\[
    A
    =
    -\cot\left(\frac{\pi}{4}-t\right)P_t
    +
    \cot\left(\frac{\pi}{4}+t\right)(\Id-P_t),
\]
the off-diagonal part of $\nabla_\nu A$ is
\[
    (\nabla_\nu A)_{\mathrm{off}}
    =
    -\left(
    \cot\left(\frac{\pi}{4}-t\right)
    +
    \cot\left(\frac{\pi}{4}+t\right)
    \right)\nabla_\nu P_t.
\]
By Lemma \ref{lem:gp-radial-curvature-identity}, $B_\nu|_{T\Sigma_t}=\Id$,
hence $B_\nu$ is block diagonal. Also $A^2$ is block diagonal. Taking the
off-diagonal part of the Riccati equation $\nabla_\nu A+A^2+B_\nu=0$ gives
\[
    -\left(
    \cot\left(\frac{\pi}{4}-t\right)
    +
    \cot\left(\frac{\pi}{4}+t\right)
    \right)\nabla_\nu P_t=0.
\]
Since the coefficient is positive, $\nabla_\nu P_t=0$.

Together with Lemma \ref{lem:gp-Codazzi-leafwise-parallelness}, this gives
$\nabla P_t=0$ on the regular normal-flow region. Hence the metric splits
locally as
\[
    g=dt^2+h_-(t)+h_+(t).
\]
Using $\partial_t g_t=2A$, we obtain
\[
    \partial_t h_-
    =
    -2\cot\left(\frac{\pi}{4}-t\right)h_-,
    \qquad
    \partial_t h_+
    =
    2\cot\left(\frac{\pi}{4}+t\right)h_+.
\]
Since
\[
    \frac{d}{dt}\log\sin^2\left(\frac{\pi}{4}-t\right)
    =
    -2\cot\left(\frac{\pi}{4}-t\right),
\]
and
\[
    \frac{d}{dt}\log\sin^2\left(\frac{\pi}{4}+t\right)
    =
    2\cot\left(\frac{\pi}{4}+t\right),
\]
we get
\[
    \frac{d}{dt}
    \left[
    \frac{h_-(t)}{\sin^2(\frac{\pi}{4}-t)}
    \right]
    =
    0,
    \qquad
    \frac{d}{dt}
    \left[
    \frac{h_+(t)}{\sin^2(\frac{\pi}{4}+t)}
    \right]
    =
    0.
\]
Because $\sin^2(\frac{\pi}{4})=1/2$, this gives
\[
    h_-(t)
    =
    2\sin^2\left(\frac{\pi}{4}-t\right)h_-^0,
    \qquad
    h_+(t)
    =
    2\sin^2\left(\frac{\pi}{4}+t\right)h_+^0.
\]
\end{proof}

\subsection{Proof of the rigidity theorem}
\label{subsec:gp-focalpoint-roundness}

We now assume that $M$ is complete. Let
\[
    \bar\Phi:\Sigma\times\RR\to M,
    \qquad
    \bar\Phi(x,t)=\exp_{F(x)}(t\nu_x),
\]
be the globally defined normal exponential map. On
$\Sigma\times(-\pi/4,\pi/4)$, it agrees with $\Phi$.

After the orientation choice in Lemma \ref{lem:gp-orientation-choice}, let
$E_-$ and $E_+$ denote the $-1$- and $+1$-eigenspaces of $A(0)$, and set
$h_-^0=g_0|_{E_-}$ and $h_+^0=g_0|_{E_+}$. With $s=\pi/4-t$, Lemma
\ref{lem:gp-normal-persistence-warped} gives
\begin{equation}\label{eq:gp-round-open-join-metric}
    g
    =
    ds^2
    +
    2\sin^2s\,h_-^0
    +
    2\cos^2s\,h_+^0,
    \qquad
    0<s<\frac{\pi}{2}.
\end{equation}

\begin{proposition}
\label{prop:gp-focalpoint-roundness}
Assume that $M$ is complete. Then $2h_-^0$ is locally the unit round metric
along the $E_-$-leaves, and $2h_+^0$ is locally the unit round metric along
the $E_+$-leaves. Consequently, locally near $F(\Sigma)$, the hypersurface
$F$ is congruent to the standard Clifford hypersurface
\[
    S^p\left(\frac1{\sqrt2}\right)
    \times
    S^{m-p}\left(\frac1{\sqrt2}\right)
    \subset S^{m+1}(1).
\]
\end{proposition}

\begin{proof}
Define
\[
    f_+=\bar\Phi\left(\cdot,\frac{\pi}{4}\right),
    \qquad
    f_-=\bar\Phi\left(\cdot,-\frac{\pi}{4}\right).
\]
For $X\in E_-$ and $U\in E_+$, the metric formula gives
\[
    |d\bar\Phi_{(x,t)}X|^2
    =
    2\sin^2\left(\frac{\pi}{4}-t\right)h_-^0(X,X),
    \qquad
    |d\bar\Phi_{(x,t)}U|^2
    =
    2\sin^2\left(\frac{\pi}{4}+t\right)h_+^0(U,U).
\]
Taking $t\to\pi/4$ gives $\ker df_+=E_-$ and
$|df_+(U)|^2=2h_+^0(U,U)$ for $U\in E_+$. Hence
$\operatorname{rank}df_+=m-p$. Similarly, $\ker df_-=E_+$ and
$\operatorname{rank}df_-=p$. Thus the images of $f_+$ and $f_-$ are locally
immersed sheets.

We prove the roundness of $2h_-^0$; the proof for $2h_+^0$ is identical.
Let $L_-$ be a connected component of a fibre of $f_+$, and write
$y=f_+(L_-)$. Let $N_+$ be the local immersed sheet of $f_+(\Sigma)$ through
$y$ determined by this fibre component. Define
\[
    D_+:L_-\to S(\nu_yN_+),
    \qquad
    D_+(x)=-\partial_t\bar\Phi\left(x,\frac{\pi}{4}\right).
\]
Indeed, for $U\in E_+$, the vector $df_+(U)$ is tangent to $N_+$, and
orthogonality of $\partial_t$ to the slices gives
\[
    \left\langle
    -\partial_t\bar\Phi\left(x,\frac{\pi}{4}\right),
    df_+(U)
    \right\rangle
    =
    \lim_{t\to\pi/4}
    \left\langle
    -\partial_t\bar\Phi(x,t),
    d\bar\Phi_{(x,t)}U
    \right\rangle
    =
    0.
\]
Since $\partial_t\bar\Phi$ has unit length, $D_+(x)\in S(\nu_yN_+)$.

Fix $x\in L_-$ and $X\in T_xL_-=E_-(x)$. For small $s>0$, set
\[
    \alpha_x(s)=\bar\Phi\left(x,\frac{\pi}{4}-s\right).
\]
Then $\alpha_x(0)=y$, $\alpha_x'(0)=D_+(x)$, and
$\alpha_x(s)=\exp_y(sD_+(x))$. Varying $x$ inside $L_-$ gives the Jacobi
field
\[
    J_X(s)=d\bar\Phi_{(x,\pi/4-s)}X
\]
along $\alpha_x$. Since $f_+$ is constant on $L_-$, we have $J_X(0)=0$, and
differentiating the initial velocity gives $D_sJ_X(0)=dD_+(X)$. Therefore
\[
    |J_X(s)|^2=s^2|dD_+(X)|^2+O(s^3).
\]
On the other hand, the exact metric formula \eqref{eq:gp-round-open-join-metric}
gives
\[
    |J_X(s)|^2
    =
    2\sin^2s\,h_-^0(X,X)
    =
    2s^2h_-^0(X,X)+O(s^4).
\]
Comparing the first nonzero terms and polarizing, we obtain
\[
    D_+^*g_{S^p(1)}=2h_-^0|_{L_-}.
\]
Thus $2h_-^0$ is locally the unit round metric along the $E_-$-leaves.

At the other focal point, using $s=\pi/4+t$, a connected component $L_+$ of a
fibre of $f_-$, and $D_-(x)=\partial_t\bar\Phi(x,-\pi/4)$, the same argument
gives
\[
    D_-^*g_{S^{m-p}(1)}=2h_+^0|_{L_+}.
\]
Thus $2h_+^0$ is locally the unit round metric along the $E_+$-leaves.

By Lemma \ref{lem:gp-Codazzi-leafwise-parallelness}, the splitting
$T\Sigma=E_-\oplus E_+$ is parallel with respect to $g_0$. Hence, locally,
$g_0$ is the Riemannian product of the two leaf metrics. Since
$2h_-^0$ and $2h_+^0$ are locally unit round, the middle slice is locally
\[
    \frac12g_{S^p(1)}+\frac12g_{S^{m-p}(1)}.
\]
Consequently the collar metric is locally
\[
    ds^2+\sin^2s\,g_{S^p(1)}+\cos^2s\,g_{S^{m-p}(1)},
\]
the round join metric on $S^{m+1}(1)$ away from its two focal factors. At
$s=\pi/4$, the middle slice is precisely
\[
    S^p\left(\frac1{\sqrt2}\right)
    \times
    S^{m-p}\left(\frac1{\sqrt2}\right)
    \subset S^{m+1}(1).
\]
This proves the local Clifford congruence.
\end{proof}

\begin{proposition}
\label{prop:gp-compact-global-conclusion}
Assume that $M$ is connected and compact. Then $(M,g_M)$ has constant
sectional curvature $1$. Hence
\[
    (M,g_M)\cong S^{m+1}(1)/\Gamma,
\]
where $\Gamma\subset O(m+2)$ is a finite group acting freely by isometries.
Moreover, the lift of $F$ to the universal cover $S^{m+1}(1)$ is locally
congruent to the standard Clifford hypersurface.
\end{proposition}

\begin{proof}
Since $M$ is compact, it is complete. By Proposition
\ref{prop:gp-focalpoint-roundness}, the open normal-flow image
\[
    \mathcal U
    =
    \bar\Phi\left(
    \Sigma\times
    \left(-\frac{\pi}{4},\frac{\pi}{4}\right)
    \right)
\]
is locally isometric to the regular part of the round join
$S^p(1)*S^{m-p}(1)$. Hence $\sec_M\equiv1$ on $\mathcal U$. We show that
$\mathcal U$ is dense in $M$.

First, every point of $M$ lies in the closed normal image
\[
    \bar\Phi\left(
    \Sigma\times
    \left[-\frac{\pi}{4},\frac{\pi}{4}\right]
    \right).
\]
Let $q\in M$. Since $M$ and $F(\Sigma)$ are compact, there exists
$x\in\Sigma$ such that $d(q,F(x))=\operatorname{dist}(q,F(\Sigma))$. Let
$\gamma:[0,\ell]\to M$ be a unit-speed minimizing geodesic from $F(x)$ to
$q$. By the first variation formula, $\gamma'(0)=\nu_x$ or
$\gamma'(0)=-\nu_x$. If $\ell\le\pi/4$, then $q$ belongs to the closed
normal image.

Suppose $\ell>\pi/4$. If $\gamma'(0)=\nu_x$, then
$d\bar\Phi_{(x,\pi/4)}(E_-(x))=0$, so $\gamma(\pi/4)$ is a focal point of
$F(\Sigma)$ along $\gamma$. If $\gamma'(0)=-\nu_x$, the same conclusion holds
at $t=-\pi/4$ using the $E_+$-directions. In either case, the minimizing
geodesic from $F(\Sigma)$ to $q$ contains a focal point before its endpoint,
contradicting the Morse index theorem for geodesics with initial endpoint
constrained to a submanifold and final endpoint fixed. Hence $\ell\le\pi/4$.

The complement of $\mathcal U$ in the closed normal image is contained in
$f_+(\Sigma)\cup f_-(\Sigma)$. The rank computation in Proposition
\ref{prop:gp-focalpoint-roundness} gives
$\operatorname{rank}df_+=m-p$ and $\operatorname{rank}df_-=p$. Since both
ranks are strictly less than $m+1=\dim M$, the images of $f_+$ and $f_-$ have
empty interior. Since the closed normal image is all of $M$, the open set
$\mathcal U$ is dense in $M$.

The curvature tensor is continuous. Since $\sec_M\equiv1$ on the dense open
set $\mathcal U$, we get $\sec_M\equiv1$ on all of $M$. Therefore the
universal cover of $M$ is the round sphere $S^{m+1}(1)$. The deck group acts
freely and isometrically on this sphere, and is finite because $M$ is compact.
Thus
\[
    (M,g_M)\cong S^{m+1}(1)/\Gamma
\]
for a finite group $\Gamma\subset O(m+2)$ acting freely by isometries.

Finally, Proposition \ref{prop:gp-focalpoint-roundness} shows that, after lifting
to the universal cover, the lifted collar is locally the round join and the
lifted hypersurface is locally the middle Clifford slice. Hence the lift of
$F$ to $S^{m+1}(1)$ is locally congruent to the standard Clifford
hypersurface.
\end{proof}

\begin{proof}[Proof of Theorem \ref{thm:gp-rigidity}]
First note that Assumption \ref{ass:gp-Ric-p} implies $\Ric_M\ge m$.  Since $M$ is connected and complete, Bonnet--Myers implies
that $M$ is compact. The first conclusion follows from Proposition
\ref{prop:gp-compact-global-conclusion}. The second conclusion is the final
assertion of the same proposition.
\end{proof}
\section{Quotient lifting and projective spaces}
\label{sec:projective-spaces}

The Betti-number hypothesis in Theorem \ref{thm:gp-focal-estimate} is not
stable under quotients. A hypersurface downstairs may have no useful
intermediate cohomology, while its natural lift to the total space may acquire
such cohomology. This gives a simple mechanism for removing the Betti-number
assumption from the downstairs statement.

We first record the general quotient-lifting principle. The projective-space
estimates then follow from the Hopf fibrations. In the complex case, focal
sets of real hypersurfaces in projective space were studied classically by
Cecil--Ryan \cite{CecilRyan1982}. Here the mechanism is different: we lift to
the total space, apply the hypersurface theorem there, and descend the
conclusion.

\begin{theorem}
\label{thm:quotient-lifting}
Let $E^{2N+1}$ satisfy Assumptions \ref{ass:gp-ambient-p-Bochner} and
\ref{ass:gp-Ric-p} with \(m=2N\), for every $1\le p\le N$. A sufficient curvature condition
is
\[
    \mathcal R_E\ge \Id_{\Lambda^2TE}.
\]
Let a compact Lie group $G$, with $d=\dim G\ge1$, act freely and
isometrically on $E$, and give $M=E/G$ the quotient metric, so that
the quotient map \(\pi:E\to M\) is a Riemannian submersion. If
$F:\Sigma^{2N-d}\to M$ is a closed connected immersed hypersurface, then
\[
    r_f(F,M)\le\frac{\pi}{4}.
\]
For a one-sided immersion, $r_f$ is understood on the two-sided normal cover.
If equality holds and \(E\) is complete, then on the connected component of
\(E\) containing the lifted component, the equality case of
Theorem \ref{thm:gp-rigidity} gives a local spherical space-form model, and the
lifted hypersurface is locally congruent to a Clifford hypersurface.
\end{theorem}

\begin{proof}
Passing to the two-sided normal cover does not change the focal radius, so we
may assume that $F$ has a global unit normal $\nu$. Since $\pi:E\to M$ is a
submersion, the fibre product
\[
    \widehat\Sigma=F^*E
    =
    \{(x,z)\in\Sigma\times E:\ F(x)=\pi(z)\}
\]
is a closed smooth manifold and a principal $G$-bundle over $\Sigma$. The map
$\widehat F:\widehat\Sigma\to E$, $\widehat F(x,z)=z$, is an immersion:
if $d\widehat F(v,w)=0$, then $w=0$, and the relation
$dF(v)=d\pi(w)$ gives $dF(v)=0$, hence $v=0$. Since
$\dim\widehat\Sigma=\dim\Sigma+d=2N$, the lift
$\widehat F:\widehat\Sigma\to E^{2N+1}$ is a closed immersed hypersurface.

At $\hat x=(x,z)$, the image of $d\widehat F$ is
\[
    \mathcal V_z\oplus \widehat{dF_x(T_x\Sigma)},
    \qquad
    \mathcal V_z=\ker d\pi_z,
\]
where the hat denotes horizontal lift. Thus the horizontal lift
$\widehat\nu$ of $\nu$ is a unit normal to $\widehat F$.

We next claim that focal radii agree. Let the projection map \(p:\widehat\Sigma\to\Sigma,\  p(x,z)=x,\)   and define
\[
    \Phi_t(x)=\exp_{F(x)}(t\nu_x),\qquad
    \widehat\Phi_t(x,z)=\exp_z(t\widehat\nu_{(x,z)}).
\]
Fix \(\hat x=(x,z)\), and write \(\widehat\gamma(t)=\widehat\Phi_t(\hat x)\) and \(\gamma(t)=\Phi_t(x)\). Since the \(G\)-action is isometric, the fundamental fields \(X^\#\),
\(X\in\mathfrak g\), are Killing. Hence the function \( t\mapsto \left\langle
        \widehat\gamma'(t),X^\#_{\widehat\gamma(t)}
    \right\rangle\) is constant. It vanishes at \(t=0\), because
\(\widehat\gamma'(0)=\widehat\nu_{\hat x}\) is horizontal and \(X^\#_z\) is
vertical. Since the vectors \(X^\#\) span the vertical distribution,
\(\widehat\gamma\) remains horizontal. Therefore, by the standard property of
Riemannian submersions, \(\widehat\gamma\) projects to the downstairs normal
geodesic \(\gamma\). Equivalently,
\[
    \pi\circ\widehat\Phi_t=\Phi_t\circ p .
\]
Differentiating gives the commutative diagram
\[
\begin{array}{ccc}
T_{\hat x}\widehat\Sigma
& \xrightarrow{\ d\widehat\Phi_t\ }
& T_{\widehat\gamma(t)}E
\\[0.3em]
dp\downarrow
& &
\downarrow d\pi
\\[0.3em]
T_x\Sigma
& \xrightarrow{\ d\Phi_t\ }
& T_{\gamma(t)}M .
\end{array}
\]
Note that \(\ker dp_{\hat x}
    =
    \{(0,X^\#_z):X\in\mathfrak g\}\), 
and the equivariance of \(\widehat\Phi_t\) gives \( d\widehat\Phi_t(0,X^\#_z)
    =
    X^\#_{\widehat\gamma(t)}\). Since the \(G\)-action is free, \(X\mapsto X^\#_q\) is an isomorphism from
\(\mathfrak g\) onto \(\mathcal V_q=\ker d\pi_q\) for every \(q\in E\). The above facts imply
\[
    d\widehat\Phi_t:
    \ker dp_{\hat x}
    \longrightarrow
    \ker d\pi_{\widehat\gamma(t)}
\]
is an isomorphism. It then follows from the commutative diagram that
\(d\widehat\Phi_t\) has nontrivial kernel if and only if \(d\Phi_t\) has
nontrivial kernel. Thus the focal times of \(\widehat F\) along
\(\widehat\nu\) agree with the focal times of \(F\) along \(\nu\), and therefore
\[
    r_f(\widehat F,E)=r_f(F,M).
\]

It remains to verify the Betti-number hypothesis upstairs. Let
$\widehat\Sigma_0$ be a connected component of $\widehat\Sigma$. Since
$d=\dim G\ge1$, a nonzero element $X\in\mathfrak g$ gives a nowhere-vanishing
fundamental vector field $X^\#$ tangent to $\widehat\Sigma_0$. Poincar\'e-Hopf then implies
$\chi(\widehat\Sigma_0)=0$. If $\widehat\Sigma_0$ is orientable, Poincar\'e
duality gives
\[
    \chi(\widehat\Sigma_0)
    =
    2+\sum_{q=1}^{2N-1}(-1)^q b_q(\widehat\Sigma_0;\RR).
\]
Thus not all intermediate Betti numbers can vanish. By Poincar\'e duality
again, there exists $p$ with $1\le p\le N$ such that
$b_p(\widehat\Sigma_0;\RR)\neq0$. If \(\widehat\Sigma_0\) is nonorientable, pass to its orientable double cover.
The lifted nonvanishing vector field still gives Euler characteristic zero, and
the pullback immersion has the same focal radius. The orientable argument then
applies to the cover.

Applying Theorem \ref{thm:gp-focal-estimate} componentwise, and after passing
to the orientable double cover if necessary, gives
\[
    r_f(F,M)=r_f(\widehat F,E)\le\frac{\pi}{4}.
\]
If equality holds and $E$ is complete, the local rigidity conclusion follows
from the equality case of the hypersurface theorem applied to the lifted
component. Finally, $\mathcal R_E\ge\Id_{\Lambda^2TE}$ implies the two required
curvature assumptions for every $p$ by Example
\ref{ex:curvature-operator-one}.
\end{proof}

We now apply the lifting principle to the Hopf fibrations
\[
    S^1\longrightarrow S^{2n+1}(1)\longrightarrow \CP^n,
    \qquad
    S^3\longrightarrow S^{4n+3}(1)\longrightarrow \HP^n.
\]
The projective metrics are normalized as the quotient metrics induced by the
unit round spheres, so that these maps are Riemannian
submersions.

\begin{theorem}
\label{thm:projective-spaces}
Let $P=\CP^n$ or $\HP^n$ with the Hopf normalization. If
$F:\Sigma\to P$ is a closed connected immersed hypersurface, then
\[
    r_f(F,P)\le\frac{\pi}{4}.
\]
For a one-sided immersion, $r_f$ is understood on the two-sided normal cover.
\end{theorem}

\begin{proof}
For $P=\CP^n$ or $\HP^n$, apply Theorem
\ref{thm:quotient-lifting} to the round sphere with $G=S^1$ or $G=S^3$. The
curvature assumptions hold upstairs by Example \ref{ex:round-space-forms}.
\end{proof}

The equality case in complex and quaternionic projective space follows by
combining the rigidity theorem upstairs with Hopf invariance of the lifted
Clifford model.

\begin{theorem}
\label{thm:projective-rigidity}
Let $P=\CP^n$ or $\HP^n$, with the Hopf normalization, and let
$F:\Sigma\to P$ be a closed connected immersed hypersurface. If
\[
    r_f(F,P)=\frac{\pi}{4},
\]
then $F$ is locally congruent to a projective Clifford hypersurface. More
precisely, after an isometry of $P$, the local model is
\[
    \left\{
    [u:v]\in
    \mathbb P_{\K}(\K^{k+1}\oplus \K^{n-k}):
    \|u\|=\|v\|
    \right\},
    \qquad
    \K=\CC\text{ or }\HH,
\]
for some $0\le k\le n-1$. Equivalently, it is the tube of radius $\pi/4$
around a totally geodesic $\KP^k\subset\KP^n$. For a one-sided immersion, \(r_f\) is understood on the two-sided normal cover.
\end{theorem}

\begin{proof}
Let $\pi:S^{\ell(n+1)-1}(1)\to\KP^n$, where $\ell=\dim_{\RR}\K$, and let
$\widehat\Sigma=F^*S^{\ell(n+1)-1}$. By the proof of Theorem
\ref{thm:projective-spaces},
\[
    r_f(\widehat F,S^{\ell(n+1)-1})
    =
    r_f(F,P)
    =
    \frac{\pi}{4}.
\]
The equality case upstairs gives a local congruence with the standard Clifford
hypersurface in the unit sphere. Equivalently, after an ambient orthogonal
isometry of the real Euclidean space underlying \(\K^{n+1}\), the lifted local
model is
\[
    \left\{
    x=x_-+x_+:
    x_\pm\in V_\pm,\ 
    \|x_-\|=\|x_+\|=\frac1{\sqrt2}
    \right\}
    =
    S(V_-,1/\sqrt2)\times S(V_+,1/\sqrt2)
    \subset S(\K^{n+1},1),
\]
for a real orthogonal splitting
\(\K^{n+1}=V_-\oplus V_+\). Here \(S(V,r)=\{v\in V:\|v\|=r\}\).

Let \(B\) denote one of the standard Hopf generators: \(B=J\) in the complex
case, and \(B=I_\alpha\), \(\alpha=1,2,3\), in the quaternionic case. Since
\(\widehat F(\widehat\Sigma)\) is locally the inverse image of \(F(\Sigma)\)
under the Hopf map, it contains the Hopf-fibre directions. Hence \( By\in T_y\widehat F(\widehat\Sigma)\) 
for every point \(y\) of the lifted local image.

Now let \(Q\) be the ambient sphere isometry sending this local lifted image to
the Clifford product
\[
    C=S(V_-,1/\sqrt2)\times S(V_+,1/\sqrt2).
\]
For \(u=Qy\in C\), set \(A=QBQ^{-1}\). Then
\[
    Au=QBQ^{-1}u=Q(By)\in T_uC.
\]
Thus the conjugated Hopf generator \(A\) is tangent to the local Clifford
product. In the quaternionic case, the same argument applies to each
\(A_\alpha=QI_\alpha Q^{-1}\).

Writing \(u=u_-+u_+\) with \(u_\pm\in V_\pm\) and
\(\|u_-\|=\|u_+\|=1/\sqrt2\), the condition \(Au\in T_uC\) gives \(\langle Au,u_-\rangle=0\) and \(\langle Au,u_+\rangle=0\). Since \(A\) is skew-symmetric, this gives \(\langle Au_+,u_-\rangle=0\) on a nonempty open subset of
\(S(V_-,1/\sqrt2)\times S(V_+,1/\sqrt2)\). By bilinearity, the same identity
holds for all \(u_-\in V_-\) and \(u_+\in V_+\). Thus
\(AV_+\subset V_+\), and by skew-symmetry also \(AV_-\subset V_-\).

For \(\K=\CC\), this says that the splitting is complex linear. For
\(\K=\HH\), applying the same argument to the three conjugated quaternionic Hopf
generators shows that the splitting is quaternionic linear. Pulling back by
\(Q^{-1}\), the original lifted Clifford splitting is invariant under the
original Hopf structure.

After an isometry, $V_-=\K^{k+1}$ and $V_+=\K^{n-k}$ for some
$0\le k\le n-1$. Taking the Hopf quotient gives the displayed projective
Clifford model. Finally, for
$L=\KP^k=\{[u:0]\}\subset\KP^n$, the Hopf normalization gives
\[
    d([u:v],L)=\arctan\frac{\|v\|}{\|u\|}.
\]
Thus $\|u\|=\|v\|$ is equivalent to $d([u:v],L)=\pi/4$.
\end{proof}

\bigskip


\end{document}